\documentclass{amsart}
\usepackage{amsmath}
\usepackage{paralist}
\usepackage{amsfonts}
\usepackage{amssymb}
\usepackage{amsthm}
\usepackage{amscd}
\usepackage{float}
\usepackage{tikz}
\usepackage{graphicx}
\usepackage[colorlinks=true]{hyperref}
\hypersetup{urlcolor=blue, citecolor=red}
\usepackage{hyperref}

\newtheorem{theorem}{Theorem}[section]

\newtheorem{proposition}[theorem]{Proposition}

\theoremstyle{definition}
\newtheorem{definition}[theorem]{Definition}
\newtheorem{remark}[theorem]{Remark}

\newcommand{\T}{\mathbb{T}}
\newcommand{\R}{\mathbb{R}}
\newcommand{\Z}{\mathbb{Z}}

\newcommand{\C}{\mathbb{C}}

\newcommand{\set}[2]{\{#1 : #2 \}}

\DeclareMathOperator*{\slim}{s-lim}

\makeatletter
\@namedef{subjclassname@2020}{%
  \textup{2020} Mathematics Subject Classification}
\makeatother

\calclayout

\begin{document} 

\title[Fourier decay in two dimensions]{Fourier transform of surface--carried measures of two-dimensional generic surfaces and applications}
 \author[J.-C.\ Cuenin]{Jean-Claude Cuenin}
 \address[J.-C.\ Cuenin]{Department of Mathematical Sciences, Loughborough University, Loughborough,
 Leicestershire, LE11 3TU United Kingdom}
 \email{J.Cuenin@lboro.ac.uk}
 \author[R. Schippa]{Robert Schippa}
\address[R. Schippa]{Karlsruher Institut f\"ur Technologie, Fakult\"at f\"ur Mathematik, Institut f\"ur Analysis, Englerstrasse 2, 76131 Karlsruhe, Germany}
\email{robert.schippa@kit.edu}

\begin{abstract}
We give a simple proof of the sharp decay of the Fourier-transform of surface-carried measures of two-dimensional generic surfaces. The estimates are applied to prove Strichartz and resolvent estimates for elliptic operators whose characteristic surfaces satisfy the generic assumptions. We also obtain new results on the spectral and scattering theory of discrete Schr\"odinger operators on the cubic lattice. 
\end{abstract}

\maketitle

\section{Introduction and main results}

The purpose of this note is to elaborate on the decay of the Fourier transform of compactly supported surface-carried measures in the two-dimensional smooth case. We aim to cover the generic behavior. The starting point is a celebrated theorem of Whitney \cite{MR0073980}, which implies that fold and cusp singularities of the Gauss map (see Definition 2.1) are the only generic singularities. Recall that in the regular case the Van der Corput lemma describes the sharp decay: Let $\Phi \in C^\infty(\R^2 ; \R)$ and $a \in C^\infty_c(\R^2)$ with $D^2 \Phi$ non-degenerate within the support of $a$, and $0$ is the only point in the support of $a$ such that $\nabla \Phi = 0$. Then,
\begin{equation}
\label{eq:VanderCorputIntroduction}
|\int_{\R^2} e^{i \lambda \Phi(x)} a(x) dx | \leq C \lambda^{-1}
\end{equation}
for $\lambda \geq 1$ with implicit constant depending only on lower eigenvalue bounds of $D^2 \Phi$, $\| \Phi \|_{C^N}$, and $\text{supp} (a)$, $\| a \|_{C^N}$ for some $N$.
We refer to Alazard--Burq--Zuily \cite{MR3637937} for a precise dependence of the implicit constant, also in higher dimensions, and to Oh--Lee \cite{OhLee2020} for further improvements.

This oscillatory integral estimate settles the decay of the Fourier transform of surface-carried measures for compact surfaces with non-degenerate Gaussian curvature like the sphere. Another classical consequence regarding the role of principal curvature is due to Littman \cite{MR0155146} (see also Greenleaf \cite{MR620265}): Let $d \mu = \beta d\sigma$ denote a surface measure $d\sigma$ of a smooth surface $\Sigma$ with compactly supported smooth density $\beta$. Littman showed that if $\Sigma$ has at least $k \geq 1$ principal curvatures bounded from below, the Fourier transform of $\mu$ satisfies the decay behavior
\begin{equation}
\label{eq:DecayFourierTransform}
|\mu^\vee(x)| = | \int_{\Sigma} e^{i x. \xi} \beta(\xi) d\sigma(\xi) | \lesssim (1+|x|)^{-\frac{k}{2}}
\end{equation}
with implicit constant depending on the lower bounds of the non-vanishing principal curvatures in modulus and $\beta$. 
However, this bound is rarely optimal except the surface is completely flat in the vanishing curvature direction. The generic sharp decay rate in \eqref{eq:DecayFourierTransform} is $k = \frac{3}{2}$. This can be proved using tools of singularity theory, see e.g.\ Duistermaat \cite{MR405513}, Varchenko \cite{Varchenko1976}, or the monograph \cite{MR2896292} of Arnold--Guse{\u\i}n-Zade--Varchenko. These methods do not always give  uniform bounds as in \eqref{eq:DecayFourierTransform}; stability results leading to uniform bounds were proved by Karpushkin \cite{MR778884}. Ikromov--M\"uller \cite{MR2854839} proved stability results under linear perturbations, which would suffice here. In fact, there is a classification of hypersurfaces in $\R^3$ in terms of the so-called linear height of the Newton polygon corresponding to the graph function defining the surface locally \cite{MR3524103}. By an application of Greenleaf's result \cite[Theorem 1]{MR620265}, the decay estimate \eqref{eq:DecayFourierTransform} implies $L^p\to L^2$ Fourier restriction bounds. However, sharp decay does not always imply sharp (in terms of optimality of $p$) restriction bounds; this is only the case if the coordinates are 'adapted'. We refer to \cite{MR2653054,MR2775788} and especially \cite{MR3524103} for a more in-depth discussion.
Our aim in this paper is to give an elementary proof of \eqref{eq:DecayFourierTransform} with $k = \frac{3}{2}$ in the generic two-dimensional case.

In the context of applications to Strichartz estimates, we also refer to decay estimates due to Ben-Artzi--Koch--Saut \cite{MR2015408}. In the work \cite{MR2015408}, after reducing to normal forms for cubic dispersion relations, the decay was made precise by making use of special functions. Our proof only makes use of one-dimensional versions of the Van der Corput lemma and is not restricted to cubic dispersion relations. Recently, Palle \cite{MR4229613}, extending the results of Ikromov--M\"uller, obtained mixed norm Strichartz estimates for general hypersurfaces in three dimensions, not only generic ones. 

In the generic case, Erd\H{o}s--Salmhofer \cite{MR2324803} derived the sharp decay of the Fourier transform of the surface measure up to logarithmic factors; however, the decay is not uniform in all directions. This result was applied to improve on the four-denominator estimate for the discrete Schr\"odinger operator on $\Z^3$ (see e.g.\ \cite{MR2333778}). Taira \cite{Taira2020B} obtained the sharp decay and proved uniform resolvent estimates for this operator. His proof uses the results of Ikromov--M\"uller and consists in evaluating certain Newton polyhedra. In Subsection \ref{Subsection:discrete} we give a simple proof of the sharp decay that only relies on (part of the) calculations already contained in the paper of Erd\H{o}s--Salmhofer. 

The surfaces under consideration $\Sigma \in \mathcal{C}$ can be locally parametrized as graphs $\{(u,v,h(u,v)): \; (u,v) \in B(0,\varepsilon) \}$ with $h(u,v) = v^2 \pm u^k + f(u,v)$ with $k=2,3,4$ and $f(u,v) = \mathcal{O}(|v|^3 + |v|^2 |u| + |u|^{k+1})$. Furthermore,
\begin{equation*}
|\partial^\alpha f(u,v)| \leq C_\alpha \min(\varepsilon^{3-|\alpha|},1).
\end{equation*}
Note that the Hessian of $h$ is degenerate at the origin if $k>2$; one says that $h$ has $A_{k-1}$ singularities. In codimension $3$ the $A_1,A_2,A_3$ singularities are the only generic ones (see \cite{MR2896292} for details). In Section \ref{section:Preliminaries} we will give a coordinate-free definition of the class $\mathcal{C}$.

We consider the surface $\{ (u,v,h(u,v)): \, (u,v) \in B(0,\varepsilon) \}$ with Fourier transform of the surface-carried measure given by
\begin{equation*}
\mu^\vee (x) = \int e^{i (x_1 u + x_2 v + x_3 h(u,v))} a(u,v) \sqrt{1+|\nabla h(u,v)|^2} du dv
\end{equation*}
with $a \in C^\infty_c(B(0,\varepsilon))$. We write $\beta(u,v) = a(u,v) \sqrt{1+|\nabla h(u,v)|^2}$.

This leads us to the analysis of oscillatory integrals
\begin{equation*}
I_{\Phi}(\lambda) = \int e^{i \lambda \Phi(u,v)} \beta(u,v) du dv
\end{equation*}
with $\beta \in C^\infty_c(B(0,\varepsilon))$ and $\Phi: \R^2 \to \R$ being a linear perturbation of $h(u,v)$.

In the first step, we give a simple proof of the decay of $I_\Phi$ only relying on variants of the one-dimensional Van der Corput lemma. The decay is essentially well-known in the literature; see, e.g., Karpushkin \cite{MR778884}. 

\begin{theorem}
\label{thm:OscillatoryIntegralEstimate}
Let $0<\varepsilon\ll 1$, $\lambda \geq 1$, $\beta: \R^2 \to \R$, and $h:\R^2 \to \R$ as above. Let
\begin{equation*}
I_{\Phi}(\lambda) = \int e^{i \lambda \Phi(u,v)} \beta(u,v) du dv, \quad \Phi(u,v) = \big( \frac{x_1}{\lambda} u + \frac{x_2}{\lambda} v + h(u,v) \big).
\end{equation*}

Then, we find the following estimate to hold
\begin{equation}
\label{eq:Decay}
|I_{\Phi}(\lambda)| \leq C \lambda^{-\frac{1}{2}-\frac{1}{k}}
\end{equation}
with $C$ uniform in $|x_i| \leq 10 \varepsilon \lambda$ and depending only on $\text{supp}(\beta)$ and $\| \beta \|_{C^N}$ for fixed $N$.
\end{theorem}

As consequence of the established decay, we can show Strichartz estimates for homogeneous dispersion relations with characteristic surface in the class $\mathcal{C}$, i.e., in the generic case. However, in the homogeneous case, cusps cannot occur: The sets with vanishing curvature become straight lines emanating from the origin and the homogeneity would stretch out the cusp along the whole line; but cusps are isolated in the generic case (see \cite{MR0073980}). We have the following:
\begin{theorem}
\label{thm:GenericStrichartzEstimates}
Let $\mu > 0$ and $p: \R^2 \backslash 0 \to \R$ be a $\mu$-homogeneous smooth function, i.e., $p(\lambda \xi) = \lambda^\mu p(\xi)$ for $\lambda > 0$. Suppose that $\{ (\xi, p(\xi)): \, \frac{1}{2} < |\xi| < 2 \} \in \mathcal{C}$. Let $u$ solve the linear dispersive equation
\begin{equation}
\label{eq:LinearDispersiveEquation}
\left\{ \begin{array}{cl}
i \partial_t u + p(\nabla / i) u &= 0, \quad (t,x) \in \R \times \R^2, \\
u(0) &= u_0 \in \mathcal{S}'(\R^2).
\end{array} \right.
\end{equation} 
Then, we find the following estimate to hold:
\begin{equation}
\label{eq:StrichartzEstimates}
\| u \|_{L^p(\R;L^q(\R^2))} \lesssim \| |D|^s u_0 \|_{L^2(\R^2)}
\end{equation}
provided that $p,q \geq 2$, $\frac{1}{p} + \frac{5}{6q} \leq \frac{5}{12}$, and $s = 1 - \frac{2}{q} - \frac{\mu}{p}$.
\end{theorem}

Next, we consider lower order perturbations
\begin{equation}
\label{eq:LOT}
p(\xi) = p_{\mu}(\xi) + p_{\nu}(\xi), \quad p_{\mu}, p_{\nu} \in C^\infty(\R^2 \backslash 0)
\end{equation}
with $p_\mu(\lambda \xi)= \lambda^\mu p_\mu(\xi)$ for some $\mu > 0$ and any $\lambda > 0$ and $|p_{\nu}(\xi)| \lesssim |\xi|^{\nu}$ for some $\nu < \mu$ as $|\xi| \to \infty$.
\begin{theorem}
\label{thm:LowerOrderPerturbation}
Let $T > 0$ and $\frac{1}{p} + \frac{3}{4q} \leq \frac{3}{8}$, and suppose that $p$ is as in \eqref{eq:LOT} with $\{(\xi,p(\xi)) : \frac{1}{2} < |\xi| < 2 \} \in \mathcal{C}$. Then, we find the following estimate to hold:
\begin{equation}
\label{eq:LocalStrichartzEstimate}
\| e^{it p(\nabla / i)} u_0 \|_{L^p([0,T],L^q(\R^2))} \lesssim_T \| u_0 \|_{H^s}
\end{equation}
with $s = 1 - \frac{2}{q} - \frac{\mu}{p}$.
\end{theorem}

As second application, we show uniform resolvent estimates for elliptic differential operators:
We consider partial differential operators
\begin{equation}
\label{eq:DifferentialOperator}
P(D) = p(-i \nabla_x)
\end{equation}
such that for $u \in \mathcal{S}'(\R^3)$ we have
\begin{equation*}
\mathcal{F}( P(D) u) (\xi) = p(\xi) \hat{u}(\xi).
\end{equation*}
By $\alpha$-ellipticity we mean that there is $\alpha > 0$ and $R>0$ such that for $|\xi| \geq R >0$
\begin{equation}
\label{eq:Ellipticity}
|p(\xi)| \gtrsim |\xi|^\alpha.
\end{equation}

\begin{theorem}
\label{thm:LimitingAbsorptionPrinciple}
Let $P(D)$ be an $\alpha$-elliptic differential operator and suppose that $\{ p(\xi) = 0 \} \in \mathcal{C}$. Then, there exists a distributional solution $u \in L^q(\R^3)$ such that
\begin{equation*}
P(D) u = f
\end{equation*}
for $f \in L^p(\R^3)$, which satisfies the estimate
\begin{equation*}
\| u \|_{L^q(\R^3)} \lesssim \| f \|_{L^{p_1}(\R^3) \cap L^{p_2}(\R^3)}
\end{equation*}
provided that
$(\frac{1}{p_1},\frac{1}{q}) \in [0,1]^2$ with
\begin{equation*}
\frac{1}{p_1} > \frac{7}{10}, \quad \frac{1}{q} < \frac{3}{10}, \quad \frac{1}{p_1} - \frac{1}{q} \geq \frac{4}{7}
\end{equation*}
and for $\alpha \leq 3$,
\begin{equation*}
0 \leq \frac{1}{p_2} - \frac{1}{q} \leq \frac{\alpha}{3}, \quad \big( \frac{1}{q}, \frac{1}{p_2} \big) \notin 
\{(0,\frac{\alpha}{3}), \; (1-\frac{\alpha}{3},1)\}.
\end{equation*}
\end{theorem}

Our next application concerns existence and completeness of wave operators for discrete Schr\"odinger operators 
\begin{align*}
H=-\frac{1}{2}\Delta+V\quad \mbox{on}\quad \ell^2(\Z^3),
\end{align*}
with deterministic or random potentials $V\in\ell^q(\Z^3)$.
Here 
\begin{align*}
(\Delta u)(x)=6u(x)-\sum_{|e|=1}u(x+e)
\end{align*}
is the discrete Laplacian. We first consider the free operator $H_0=-\frac{1}{2}\Delta$ and denote by $E_0(\lambda)$ its spectral measure and by $R_0(z)$ its resolvent. By slight abuse of notation we also write $R_0(\lambda)$ for any one of the (in general distinct) limiting resolvents $R_0(\lambda\pm i 0)$.

\begin{theorem}\label{theorem resolvent discrete}
Let $J$ be a compact subset of $\R\setminus\{0,2,3,4,6\}$. Then, the following estimates hold:
\begin{enumerate}
\item[(1)] Spectral measure estimate:
\begin{align*}
\sup_{\lambda\in J}\|E_0'(\lambda)\|_{\ell^{14/11}(\Z^3)\to \ell^{14/3}(\Z^3)}<\infty.
\end{align*}
\item[(2)] Uniform resolvent estimate:
\begin{align*}
\sup_{\lambda\in J}
\|R_0(\lambda)\|_{\ell^{14/11}(\Z^3)\to \ell^{14/3}(\Z^3)}<\infty.
\end{align*}
\item[(3)] H\"older continuity: For any $\delta\in (0,1]$,
\begin{align*}
\sup_{\lambda,\mu\in J}
|\lambda-\mu|^{-\beta_{\delta}}\|R_0(\lambda)-R_0(\mu)\|_{\ell^{p_{\delta}}(\Z^3)\to \ell^{p_{\delta}'}(\Z^3)}<\infty,
\end{align*}
where $\beta_{\delta}=(1/p-1)\delta$ and $1/p_{\delta}-1/p_{\delta}'=1/(7/4-\delta)$. 
\end{enumerate}
\end{theorem} 

\begin{remark}
The above estimates are stable under small perturbations $V\in \ell^{7/4}(\Z^3;\C)$. This follows from the resolvent identity and H\"older's inequality, see e.g.\ \cite{cuenin2020spectral} for similar arguments.
\end{remark}

Next, we state our result on existence and completeness of wave operators for deterministic and random potentials. 
For deterministic potentials we require $V\in \ell^{7/4}(\Z^3;\R)$. In the random case we assume that the potential $V_{\omega}(x)=\omega_xv_x$ has sub-gaussian distribution, for instance $\{\omega_x:x\in\Z^d\}$ are Bernoulli or normalized Gaussians. Bourgain \cite{MR1877824,MR2083389} showed that randomization allows to cut the decay of the potential in half. The tail distribution is important here as the argument of Bourgain uses Dudley's $L^{\psi_2}$-estimate (see e.g. \cite[8.1]{MR3837109}). The random part of the following theorem is an extension of Bourgain's \cite{MR2083389} result from the two-dimensional to the three-dimensional lattice $\Z^3$.

\begin{theorem}\label{wave operators discrete}
Assume that one of the following holds:
\begin{enumerate}
\item $V\in \ell^{7/4}(\Z^3;\R)$ is a deterministic potential. 
\item $\epsilon > 0$ is fixed and $V_{\omega}$ is a random potential such that $(1+|\cdot|)^{\epsilon}V_{\omega}\in \ell^{7/2}(\Z^3;\R)$.
\end{enumerate}
Then, the wave operators
\begin{align*}
W_{\pm}(H,H_0):=\slim_{t\to\pm\infty}e^{i t H}e^{-i t H_0}
\end{align*}
exist and are complete.
\end{theorem}

\begin{remark}
Korotyaev and M\o ller \cite{MR4018760} proved existence and completeness of wave operators in the deterministic case for $V\in \ell^{p}(\Z^3;\R)$ for $p<6/5$. Since $\ell^p$ spaces on the lattice are nested, the present result is stronger. The random case in three dimensions is completely new.
\end{remark}

\textbf{Acknowledgements} The first author would like to thank Isroil Ikromov, Orif Ibrogimov, and Kouichi Taira for useful discussions. The second author acknowledges financial support by the Deutsche Forschungsgemeinschaft (DFG, German Research Foundation) -- Project-ID 258734477 -- SFB 1173.

\section{Preliminaries}
\label{section:Preliminaries}
The purpose of this section is to define the class of considered surfaces and reduce the claimed decay estimate for the Fourier transform to the oscillatory integral estimate from Theorem \ref{thm:OscillatoryIntegralEstimate}.

We define the class\footnote{Confining to embedded $\Sigma \hookrightarrow \R^3$, the class is a set.} $\mathcal{C}$ of bounded two-dimensional smooth surfaces $\Sigma$, which will be considered in the following. Let $N: \Sigma \to \mathbb{S}^2$ denote the Gauss map of the smooth surface $\Sigma$\footnote{The crucial properties are local. For the sake of simplicity we suppose that $\Sigma$ is globally orientable.} and let $K: \Sigma \to \R$ denote the Gauss curvature. Recall that $K(p) \neq 0$ if and only if the Gauss map is a local diffeomorphism at $p$, in other words $\text{rank } dN_p = 2$. In this case we say that the Gauss map has no singularity at $p$. If
\begin{equation}
\label{eq:RankCondition}
\text{rank } (dN)_p = 1 \text{ and } d (\det dN)_p \neq 0,
\end{equation}
then the zero curvature set $\Gamma = \{ p \in \Sigma : \det dN_p = 0 \}$ is locally a smooth curve, by the implicit function theorem. Let $V$ be a non-vanishing tangent vector field to $\Gamma$. 
\begin{definition}
Assume that \eqref{eq:RankCondition} holds. If $dN(V)$ does not vanish at $p$, then we call $p$ a \emph{fold} (singularity) of the Gauss map. If $dN(V)$ vanishes linearly at $p$, then we call $p$ a \emph{cusp} (singularity). We say that $\Sigma \in \tilde{\mathcal{C}}$ if at each point the Gauss map has either a fold ($A_2$), a cusp ($A_3$), or no singularity ($A_1$). If $\Sigma \in \tilde{\mathcal{C}}$ and $\partial \Sigma = \emptyset$ or there is $\Sigma' \in \tilde{\mathcal{C}}$ with $\Sigma \hookrightarrow \Sigma'$ and $\text{dist}(\Sigma,\partial \Sigma') > 0$), then we say that $\Sigma \in \mathcal{C}$.
\end{definition}
Whitney \cite{MR0073980} showed that singularities of the Gauss map, which are neither folds nor cusps, are highly non-generic. We refer to \cite{MR0341518} and \cite{MR0494220} for further reading and illustration.

The aim of this paper is to estimate the decay of the Fourier transform
\begin{equation*}
\mu ^{\vee}(x) = \int_{\Sigma} e^{i x\cdot\xi} \beta(\xi) d\sigma(\xi), \quad \beta \in C^\infty_c(\Sigma)
\end{equation*}
as
\begin{equation}\label{eq: 3/4 decay}
|\mu^\vee(x)| \lesssim (1+|x|)^{-\frac{3}{4}}
\end{equation}
 with tractable implicit constant suitable for PDE applications. First note that it is enough to analyze a neighbourhood $\mathcal{N}$ of $\{K(p) = 0 \}$. By the above, $\{K(p) = 0 \}$ decomposes into finitely many isolated curves. Note that the curves have to be disjoint by the implicit function theorem.
 
Indeed, the estimate
\begin{equation*}
\big| \int_{\Sigma \backslash \mathcal{N}} e^{ix\cdot\xi} \beta(\xi) d\sigma(\xi) \big| \lesssim (1+|x|)^{-1}
\end{equation*}
with implicit constant depending on $\text{supp}(\beta)$, $\| \beta \|_{C^N}$ and smoothly on $\Sigma$ is a consequence of the Van der Corput lemma \cite{MR3637937,OhLee2020}. By compactness of $\Sigma$, finitely many of the above estimates yield the desired decay.

Let $\Gamma \subseteq \{ K(p) = 0 \}$ be a closed curve and $\mathcal{N}_\Gamma$ be a neighbourhood in $\Sigma$ such that for $ p \in \mathcal{N}_\Gamma$ with $K(p) = 0$ it follows that $p \in \Gamma$. Again by compactness, it suffices to prove the decay
\begin{equation*}
\big| \int_{\mathcal{N}_\Gamma} e^{ix\cdot\xi} \beta(\xi) d\sigma(\xi) \big| \lesssim (1+|x|)^{-\frac{3}{4}}
\end{equation*}
with implicit constant depending on $\text{supp}(\beta)$, $\| \beta \|_{C^N}$ and smoothly on $\mathcal{N}_\Gamma$. Note that the cusps are necessarily isolated points. It is enough to consider the case where $\beta$ is centred at a cusp singularity, with the cusp singularity being the only one in the support of $\beta$, or $\beta$ centred at a fold with no cusps in the support of $\beta$. After rigid motion and changing to graph parametrisation, we can suppose that
\begin{equation*}
\Sigma = \{(u,v,h(u,v)): (u,v) \in B(0,\varepsilon) \}
\end{equation*}
and
\begin{equation*}
\mu^\vee(x) = \int_{B(0,\varepsilon)} e^{ix.(u,v,h(u,v))} \sqrt{1+|\nabla h(u,v)|^2} \beta(u,v) du dv.
\end{equation*}
By a linear change of variables, leaving $|\mu^\vee(x)|$ invariant, we can suppose that $h(0,0) = 0$, $\nabla h(0,0) = 0$. For the sake of brevity, we also write $0 = (0,0)$ and $\nabla h(u,v) = (\partial_u h, \partial_v h)$. The Gauss map in these coordinates is given by
\begin{equation*}
N(u,v) = \frac{1}{\sqrt{1+ |\nabla h(u,v)|^2}} (-h_u, h_v, 1)
\end{equation*}
and the derivative is given by
\begin{equation}
\label{eq:DerivativeGaussMap}
-dN (u,v) = \frac{1}{\sqrt{1+ h_u^2 + h_v^2}} \begin{pmatrix}
h_{uu} & h_{uv} \\
h_{uv} & h_{vv}
\end{pmatrix}
.
\end{equation}
In the following we recall how $h$ can be expressed as perturbation of normal forms depending on a fold or cusp singularity at the origin. After a linear change of coordinates, depending smoothly on $\Sigma$, we may assume by the first condition in \eqref{eq:RankCondition} that
\begin{equation}
\label{eq:SecondDerivatives}
h_{uu}(0) = 0, \quad h_{uv}(0) = 0, \quad h_{vv}(0) = 1.
\end{equation}
The second condition in \eqref{eq:RankCondition} implies that
\begin{equation}
\label{eq:ThirdDerivatives}
(h_{uuu}(0), h_{uuv}(0)) \neq (0,0).
\end{equation}
Recall that the subset of $\Sigma$ with vanishing Gaussian curvature is a smooth curve
\begin{equation*}
\Gamma = \{(u,v,h(u,v)): (u,v) \in B(0,\varepsilon), \; J(u,v) = 0 \} \subseteq \Sigma,
\end{equation*}
where $J = h_{uu} h_{vv} - h_{uv}^2$. By the implicit function theorem, \eqref{eq:SecondDerivatives} implies that locally the equation
\begin{equation}
\label{eq:ParametrisationZeroCurvature}
h_v(u,v) = 0
\end{equation}
has a unique smooth solution $v = \psi(u)$, that is, $h_v(u,\psi(u)) = 0$. A Taylor series expansion of $h$ with respect to the variable $v$ around $\psi(u)$ then shows that
\begin{equation}
\label{eq:Expansionh}
h(u,v) = a(u,v)(v-\psi(u))^2 + b(u)
\end{equation}
for some smooth functions $a$,$b$ such that $a(0) = \frac{1}{2} h_{vv}(0) = \frac{1}{2}$ and $b(u) = \mathcal{O}(u^3)$. This argument is detailed in \cite{MR3524103}. By differentiating \eqref{eq:ParametrisationZeroCurvature}, we find
\begin{equation*}
\psi(0) = \psi_u(0) = 0.
\end{equation*}
We turn to a case-by-case analysis of \eqref{eq:ThirdDerivatives}:

\textbf{Case 1:} $h_{uuu}(0) \neq 0$: We have $[dN(V)](0) = h_{uuu}(0) \neq 0$, which means that the singularity is a fold. Also note that necessarily $b_{uuu}(0) \neq 0$. By the nonlinear change of variables $v \to v + \psi(u)$, which does not change the measure however, we find
\begin{equation*}
\begin{split}
h(u,v) &= a(u,v) v^2 + c u^3 + \mathcal{O}(u^4) \\
&= v^2 + cu^3 + \mathcal{O}(|v|^3 + |v|^2 |u| + |u|^4),
\end{split}
\end{equation*}
with the last line following by Taylor expansion of $a(u,v)$ and $b(u)$.

\textbf{Case 2:} $h_{uuu}(0) = 0, \; h_{uuv}(0) \neq 0$. By the implicit function theorem, $\Gamma$ can be parametrized by $u$ so that $J(u,v(u)) = 0$ on $\Gamma$. We find
\begin{equation*}
v(u) = \frac{h^2_{uuv}(0) - \frac{1}{2} h_{uuuu}(0)}{h_{uuv}(0)} u^2 + \mathcal{O}(|u|^3).
\end{equation*}
The vector field $V = (1,v_u(u))$ is tangent to $\Gamma$ and satisfies
\begin{equation*}
dN_{(u,v(u))}(V) = 
\begin{pmatrix}
0 \\ h_{uuv}(0) u + v_u(u)
\end{pmatrix}
+ \mathcal{O}(|u|^2).
\end{equation*}
This vanishes linearly at the origin if and only if
\begin{equation*}
h_{uuuu}(0) - 3h^2_{uuv}(0) \neq 0,
\end{equation*}
in which case the singularity is a cusp. We compute
\begin{equation*}
h_{uuuu}(0) = 3 \psi_{uu}(0)^2 + b_{uuuu}(0), \quad h_{uuv}(0) = -\psi_{uu}(0)
\end{equation*}
and hence, $b_{uuuu}(0) \neq 0$. After a change of variables $v \to v+ \psi(u)$, we obtain after Taylor expansion of $a$ and $b$:
\begin{equation*}
h(u,v) = a(u,v) v^2 + cu^4 + \mathcal{O}(u^5) =  v^2 + cu^4 + \mathcal{O}(|v|^3 + |v|^2 |u| + |u|^5).
\end{equation*}

In conclusion, the surface under consideration can be locally parametrized as the graph $\{(u,v,h(u,v)): \; (u,v) \in B(0,\varepsilon) \}$ with $h(u,v) = v^2 \pm u^k + f(u,v)$, $k=3,4$, and $f(u,v) = O(|v|^3 + |v|^2 |u| + |u|^{k+1})$ and furthermore,
\begin{equation*}
|\partial^\alpha f(u,v)| \leq C_\alpha \min(\varepsilon^{3-|\alpha|},1).
\end{equation*}
The Jacobians of the involved coordinate changes depend smoothly on the surface, and so do the implicit constants in the above display. Furthermore, by compactness of the surface, only finitely many of the above parametrizations have to be considered.

We turn to the Fourier decay of the surface $\{ (u,v,h(u,v)): \, (u,v) \in B(0,\varepsilon) \}$ with Fourier transform of the surface-carried measure given by
\begin{equation*}
\mu^\vee (x) = \int e^{i (x_1 u + x_2 v + x_3 h(u,v))} a(u,v) \sqrt{1+|\nabla h(u,v)|^2} du dv
\end{equation*}
with $a \in C^\infty_c(B(0,\varepsilon))$. We write $\beta(u,v) = a(u,v) \sqrt{1+|\nabla h(u,v)|^2}$.
 Note that for $|(x_1,x_2)| \geq 10 \varepsilon |x_3|$ we find by non-stationary phase estimates
\begin{equation*}
|\mu^\vee(x)| \leq C_N (1+|x|)^{-N}
\end{equation*}
for any $N \geq 1$ with $C_N=C_N(\text{supp} (\beta), \| \beta \|_{C^N})$. To estimate the non-trivial contribution $|(x_1,x_2)| < 10 \varepsilon |x_3|$, we let $x_3 = \lambda \geq 1$ and consider
\begin{equation*}
I_{\Phi}(\lambda) = \int e^{i \lambda \Phi(u,v)} \beta(u,v) du dv, \quad \Phi(u,v) = \big( \frac{x_1}{x_3} u + \frac{x_2}{x_3} v + h(u,v) \big).
\end{equation*}
In the next section we derive the bound claimed in Theorem \ref{thm:OscillatoryIntegralEstimate}
\begin{equation}
|I_{\Phi}(\lambda)| \leq C \lambda^{-\frac{1}{2}-\frac{1}{k}}
\end{equation}
with $C$ uniform in $|x_i| \leq 10 \varepsilon \lambda$ and depending only on $\text{supp}(\beta)$ and $\| \beta \|_{C^N}$ for fixed $N$.


\section{Proof of the oscillatory integral estimate}

We note that for any $|u| \leq \varepsilon$, by monotonicity of $\partial_v \Phi(u,v)$ there is $v^*$, $|v^*| \leq 10 \varepsilon$ such that
\begin{equation*}
\partial_v \Phi(u,v^*) = 0, \quad \partial_{vv} \Phi(u,v^*) = 2 - \mathcal{O}(\varepsilon).
\end{equation*}
Choosing $\varepsilon$ small enough, the implicit function theorem and the hypothesis on $f$ yield that there is exactly one smooth solution $\psi(u)$ such that
\begin{equation*}
\partial_v \Phi(u,\psi(u)) = 0.
\end{equation*}
In the first step, we use the following one-dimensional variable-coefficient version of the Van der Corput lemma:
\begin{proposition}[{\cite[Corollary~1.1.3]{MR3645429}}]
\label{prop:VariableCoefficientVdC}
Let
\begin{equation*}
\tilde{a}(u,\lambda) = e^{-i \lambda \Phi(u,\psi(u))} \int_{\R} e^{i \lambda \Phi(u,v)} a(\lambda,u,v)) dv.
\end{equation*}
Suppose that
\begin{equation*}
|\partial_\lambda^m \partial_u^{n_1} \partial_v^{n_2} a(\lambda,u,v)| \leq (1+\lambda)^{-m} d_{m,n_1,n_2}.
\end{equation*}
Then the following estimate holds,
\begin{equation}
\label{eq:VariableCoefficientVdC}
|\partial_\lambda^j \partial_u^\beta \tilde{a}(u,\lambda)| \leq c_{j,\beta} \lambda^{-\frac{1}{2}-j}.
\end{equation}
For fixed $j$ and $\beta$, $c_{j,\beta}$ depends only on the lower bound of $\Phi_{vv}$, finitely many $d_{m,n_1,n_2}$, and $|\text{supp}(a)|$.
\end{proposition}
In the second step, we use the following instance of the one-dimensional Van der Corput lemma:
\begin{proposition}[{\cite[p.~334]{MR1232192}}]
\label{prop:OneDimensionalVdC}
Suppose $\phi$ is real-valued and smooth in $(a,b)$ and that $|\phi^{(k)}(x)| \geq 1$ for all $x \in (a,b)$ with $k \geq 2$. Then
\begin{equation*}
\big| \int_a^b e^{i\lambda \phi(x)} \psi(x) dx \big| \leq c_k \lambda^{-\frac{1}{k}} \big[ \,|\psi(b)| + \int_{a}^b |\psi'(x)| dx \big]
\end{equation*}
with $c_k$ independent of $\lambda$ and $\phi$.
\end{proposition}
By this, we can apply the one-dimensional Van der Corput lemma to
\begin{equation*}
I_{\Phi}(\lambda) = \int_{\R} e^{i \lambda \Phi(u,\psi(u))} \tilde{a}(u,\lambda) du.
\end{equation*}
By assumption we have $|\Phi^{(k)} (u,\psi(u))| \geq c(k) > 0$ for $u \in \text{supp}(a)$ and by Proposition \ref{prop:VariableCoefficientVdC} we know that $\tilde{a}$ satisfies \eqref{eq:VariableCoefficientVdC} with implicit constant only depending on $\text{supp}(\beta)$ and $\| \beta \|_{C^N}$. This yields the desired estimate \eqref{eq:Decay} with the claimed dependence.
\hfill $\Box$
\begin{remark}
\label{rem:Stability}
Regarding the smooth dependence, we record the following consequence. Suppose that we have a family of surfaces $(p^{-1}(\{a\}))_{a \in (-\varepsilon,\varepsilon)}$ given as level sets of a smooth and regular function $p$. With the Gauss map smoothly depending on $a$, we find that $\mathcal{C}$ is stable under smooth perturbations. Moreover, the decay is uniform in $a$ as follows from the construction: The graph parametrisation, as well as the change of variables in the proof of the decay estimate, can be chosen to be locally smoothly varying in $a$.
\end{remark}
\section{Applications}

In this section we apply the decay estimate to derive Strichartz estimates for homogeneous phase functions and resolvent estimates for elliptic operators.
\subsection{Strichartz estimates}
Consider the solution to
\begin{equation}
\label{eq:DispersiveEquation}
\left\{ \begin{array}{cl}
i \partial_t u + p(D) u &= 0, \quad (t,x) \in \R \times \R^2, \\
u(0) &= \beta(D) u_0
\end{array} \right.
\end{equation}
with $\beta \in C^\infty_c(B(0,4) \backslash B(0,1/4))$, $\sum_{k \in \Z} \beta(2^{-k} \xi) \equiv 1$ for $\xi \neq 0$.
We can write
\begin{equation*}
u(t,x) = \int e^{i (x\cdot\xi + t p(\xi))} \beta(\xi) \hat{u}_0(\xi) d\xi = G_t * u_0(x).
\end{equation*}

In the generic case the surface
\begin{equation*}
\{ (\xi, p(\xi)) \, : \, \xi \in \text{supp}(\beta) \}
\end{equation*}
has at most cusp singularities. In this case, the decay estimate
\begin{equation}
\label{eq:KernelEstimate}
|G_t(x)| \lesssim (1+|t|)^{-\frac{3}{4}}
\end{equation}
is a consequence of Theorem \ref{thm:OscillatoryIntegralEstimate}. The homogeneous Strichartz estimates
\begin{equation}
\label{eq:GenericStrichartzEstimates}
\| u(t) \|_{L^p(\R;L^q(\R^2))} \lesssim \| u_0 \|_{L^2(\R^2)}
\end{equation}
with $p,q \geq 2$, $\frac{1}{p} + \frac{3}{4q} \leq \frac{3}{8}$ are a consequence of \cite[Theorem~1.2]{MR1646048}: Consider the family of operators $(U(t))_{t \in \R}$ on $L^2(\R^2)$, which satisfy the dispersive estimate with decay parameter $\sigma = \frac{3}{4}$ by \eqref{eq:KernelEstimate} and the energy estimate by Plancherel's theorem. Inhomogeneous estimates follow likewise. 

We turn to the proof of Theorem \ref{thm:GenericStrichartzEstimates}: If the dispersion relation is homogeneous, then cusps cannot occur and in the generic case, folds are the only singularities of the characteristic surface. Indeed, with the notations of Section \ref{section:Preliminaries}, it follows that from \eqref{eq:DerivativeGaussMap} and the homogeneity of $p(\xi)$ that if $(dN(V))_p = 0$ at some point $p$, then $(dN(V))_{\lambda p} = 0$ at all points $\lambda p$, $\lambda>0$.
This would contradict the linear vanishing of $dN(V)_p$, and hence no cusp can occur.

Now let $(P_N)_{N \in 2^{\mathbb{Z}}}$ denote a homogeneous Littlewood-Paley decomposition and $p,q$ as in Theorem \ref{thm:GenericStrichartzEstimates}. By the square function estimate and $p,q \geq 2$, we find
\begin{equation*}
\|u\|_{L^p(\R;L^q(\R^2))}\lesssim\| \big( \sum_{N \in 2^{\Z}} |P_N u(t,x) |^2 \big)^{\frac{1}{2}} \|_{L^p(\R;L^q(\R^2))} \leq \big( \sum_N \| P_N u(t,x) \|^2_{L^p(\R;L^q(\R^2))} \big)^{\frac{1}{2}}.
\end{equation*}
We use scaling $x \to Nx$, $t \to N^\mu t$, $\xi \to \frac{\xi}{N}$, and \eqref{eq:GenericStrichartzEstimates} to compute
\begin{equation*}
\| P_N u \|_{L^p(\R;L^q(\R^2))} = N^{-\frac{\mu}{p} - \frac{2}{q}} \| P_0 \tilde{u} \|_{L^p(\R;L^q(\R^2))} 
\lesssim N^{-\frac{\mu}{p} - \frac{2}{q} + 1} \| P_N u_0 \|_{L^2(\R^2)}.
\end{equation*}
The proof is concluded by almost orthogonality of the frequency localized pieces.
\hfill $\Box$

We turn to the proof for perturbations:
\begin{proof}[Proof of Theorem \ref{thm:LowerOrderPerturbation}]
Note that the dyadic surfaces $\{(N\xi, p_\mu(N \xi) + p_{\nu}(N\xi)) : \, |\xi| \sim 1 \}$ become after rescaling as previously
\begin{equation*}
(\xi, p_\mu(\xi) + \frac{p_{\nu}(N \xi)}{N^\mu}) =: (\xi,p_N(\xi)).
\end{equation*}
Hence, we have by Remark \ref{rem:Stability} $\{(\xi,p_N(\xi)) : |\xi| \sim 1 \} \in \mathcal{C}$ with uniform dispersive estimate
\begin{equation*}
\| e^{it p_N(D)} \beta(D) u_0 \|_{L^\infty} \lesssim (1+|t|)^{-\frac{3}{4}} \| u_0 \|_{L^1},
\end{equation*}
where we have to choose $N$ large enough in dependence of $p_\nu$ and $p_\mu$. We then obtain by the arguments of the previous proof
\begin{equation*}
\| e^{it p(D)} P_{\gtrsim N} u_0 \|_{L^p([0,T],L^q(\R^2))} \lesssim  \| u_0 \|_{H^s}.
\end{equation*}
The low frequencies are estimated by H\"older's and Bernstein's inequality:
\begin{equation*}
\| e^{it p(D)} P_{\lesssim N} u_0 \|_{L^p([0,T],L^q(\R^2))} \lesssim_{T,N} \| u_0 \|_{H^s}.
\end{equation*}
The proof is complete by taking the two previous estimates together.
\end{proof}

\subsection{Resolvent estimates}
Next we prove Theorem \ref{thm:LimitingAbsorptionPrinciple}. The method of proof is well-known and for details we refer e.g. to \cite{MandelSchippa2021}. In the first step, a Fourier restriction--extension theorem for surfaces $\Sigma_a$, $a \in (-\delta_0,\delta_0)$ is derived. By Remark \ref{rem:Stability}, we can suppose that $\Sigma_a:=\{p(\xi)=a\} \in \mathcal{C}$ for $a \in (-\delta_0,\delta_0)$ with uniform decay of the Fourier transform of the surface measure. We prove strong bounds
\begin{equation}
\label{eq:StrongBoundsFourierRestrictionExtension}
\| \int_{\R^3} e^{ix\cdot\xi} \delta_{\Sigma_a}(\xi) \beta(\xi) \hat{f}(\xi) d\xi \|_{L^q(\R^3)} \lesssim \| f \|_{L^p(\R^3)}
\end{equation}
within a pentagonal region. Here $\beta \in C^\infty_c$ localizes to a suitable neighbourhood of $\{K=0\}$ in $\big( \Sigma_a \big)_{a \in (-\delta_0,\delta_0)}$.
Away from $\{ K=0\}$, \cite[Theorem~1.3]{MandelSchippa2021} provides better estimates for $d=3$, $k=2$. On part of the boundary of the pentagonal region, we show weak bounds
\begin{align}
\label{eq:WeakBoundI}
\| \int_{\R^3} e^{ix\cdot\xi} \delta_{\Sigma_a}(\xi) \beta(\xi) \hat{f}(\xi) d\xi \|_{L^{q,\infty}(\R^3)} \lesssim \| f \|_{L^p(\R^3)} \\
\label{eq:WeakBoundII}
\| \int_{\R^3} e^{ix\cdot\xi} \delta_{\Sigma_a}(\xi) \beta(\xi) \hat{f}(\xi) d\xi \|_{L^q(\R^3)} \lesssim \| f \|_{L^{p,1}(\R^3)},
\end{align}
and lastly, restricted weak bounds
\begin{equation}
\label{eq:RestrictedWeakBound}
\| \int_{\R^3} e^{ix\cdot\xi} \delta_{\Sigma_a}(\xi) \beta(\xi) \hat{f}(\xi) d\xi \|_{L^{q,\infty}(\R^3)} \lesssim \| f \|_{L^{p,1}(\R^3)}
\end{equation}
at its inner endpoints. We refer to the figure below for illustration.
For $X,Y \in [0,1]^2$ we write $[X,Y] = \{ Z: \, \exists \lambda \in [0,1]: \, Z = \lambda X + (1-\lambda) Y \}$ and correspondingly $(X,Y)$, $(X,Y]$, etc.

\begin{proposition}
\label{prop:FourierRestrictionExtension}
Let $p:\R^3 \to \R$ be $\alpha$-elliptic with $\delta_0 > 0$ such that $(\Sigma_a = \{ p(\xi) = a \})_{a \in (-\delta_0,\delta_0)} \subseteq \mathcal{C}$. 
Then, we find \eqref{eq:StrongBoundsFourierRestrictionExtension} to hold for $(\frac{1}{p},\frac{1}{q}) \in [0,1]^2$ provided that
\begin{equation*}
\frac{1}{p} > \frac{7}{10}, \quad \frac{1}{q} < \frac{3}{10}, \quad \frac{1}{p} - \frac{1}{q} \geq \frac{4}{7}.
\end{equation*}
Let
\begin{equation*}
B= \big( \frac{7}{10}, \frac{9}{70} \big), \; C = \big( \frac{7}{10}, 0 \big), \quad B' = \big( \frac{61}{70}, \frac{3}{10} \big), \; C' = \big(1, \frac{3}{10} \big).
\end{equation*}
Then \eqref{eq:WeakBoundI} holds for $(1/p,1/q) \in (B',C']$, \eqref{eq:WeakBoundII} for $(1/p,1/q) \in (B,C]$, and \eqref{eq:RestrictedWeakBound} for $(1/p,1/q) \in \{B,B' \}$.
\end{proposition}
\begin{center}
\begin{figure}
\label{fig:FourierRestrictionExtensionEstimate}
\begin{tikzpicture}[scale=0.5]
\draw[->] (0,0) -- (13,0); \draw[->] (0,0) -- (0,13);
\draw (0,0) --(12,12); \draw(0,12) -- (12,12); \draw (12,12) -- (12,0);
\coordinate (E) at (6,3);
\coordinate [label=left:$\frac{3}{14}$] (EX) at (0,3);
\coordinate [label=above:$B$] (B) at (7,2.5);
\coordinate [label=above:$B'$] (B') at (9.5,5);
\coordinate [label=below:$C$] (C) at (7,0);
\coordinate [label=right:$C'$] (C') at (12,5);
\coordinate [label=left:$\frac{1}{2}$] (Y) at (0,6);
\coordinate [label=left:$1$] (YY) at (0,12);
\coordinate [label=below:$\frac{1}{2}$] (X) at (6,0);
\coordinate [label=below:$1$] (XX) at (12,0);
\coordinate [label=left:$\frac{1}{q}$] (YC) at (0,9);
\coordinate [label=below:$\frac{1}{p}$] (XC) at (9,0);

\draw [dotted] (EX) -- (E); \draw [dotted] (E) -- (X); \draw [dotted] (E) -- (XX); \draw [dotted] (XX) -- (YY);
\draw [help lines] (C) -- (B); \draw (B) -- (B');
\draw [help lines] (B') -- (C');

\foreach \point in {(C),(C'),(B),(B')}
	\fill [black, opacity = 1] \point circle (3pt);

\end{tikzpicture}
\caption{Pentagonal region, within which strong $L^p$-$L^q$-Fourier restriction extension estimates hold.}
\end{figure}
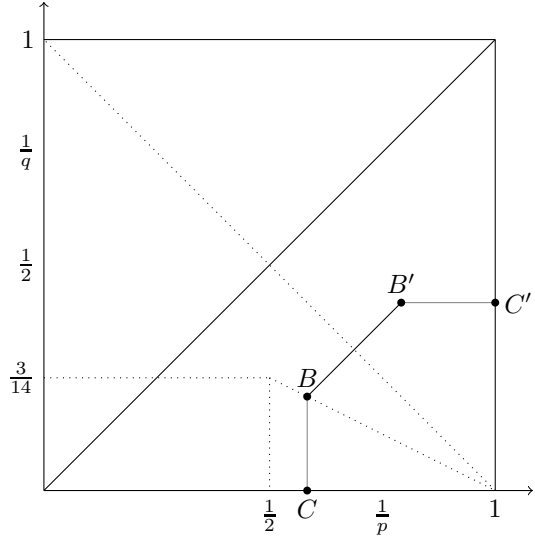
\end{center}

In the second step we foliate a neighbourhood $U$ of $\Sigma_0$ with level sets of $p$ to show bounds $\| A_\delta f \|_{L^q} \lesssim \| f \|_{L^p(\R^3)}$ for
\begin{equation}
\label{eq:SingularMultiplier}
A_\delta f(x) = \int_{\R^3} \frac{e^{ix\cdot\xi} \beta_1(\xi)}{p(\xi) + i \delta} \hat{f}(\xi) d\xi
\end{equation}
independent of $\delta$. Here, $p,q$ are as in Proposition \ref{prop:FourierRestrictionExtension} and $|p(\xi)| \leq \delta_0$ for $\xi \in \text{supp } (\beta_1)$ with $\Sigma_0 \subseteq \text{supp }(\beta_1)$. Away from the singular set, estimates for
\begin{equation}
\label{eq:ExteriorDomainMultiplier}
B_\delta f(x) = \int_{\R^3} \frac{e^{ix\cdot\xi} \beta_2(\xi)}{p(\xi) + i \delta} \hat{f}(\xi) d\xi
\end{equation}
with $\beta_1 + \beta_2 \equiv 1$ follow from Young's inequality and properties of the Bessel potential. The estimate of $\| B_\delta \|_{L^p \to L^q}$ depends on ellipticity.\\
The method of proof is well-known and detailed in \cite{MandelSchippa2021}; see also \cite{MR4076079,MR3545933,MR4153099} and references therein. Also note that the self-dual $L^p\to L^{p'}$ resolvent estimates follow from (in fact, are equivalent to) \eqref{eq:StrongBoundsFourierRestrictionExtension} by the abstract argument in \cite{cuenin2020spectral}. We omit the details to avoid repetition and turn to the conclusion of the proof of Theorem \ref{thm:LimitingAbsorptionPrinciple}, relying on Proposition \ref{prop:FourierRestrictionExtension}. The argument parallels \cite[Section~5.2]{MandelSchippa2021} very closely, to avoid repitition we shall be brief. Let $A_\delta$ and $B_\delta$ be as in \eqref{eq:SingularMultiplier} and \eqref{eq:ExteriorDomainMultiplier}. We start with the more difficult estimate of $A_\delta$. We show boundedness of $A_\delta:L^p(\R^3) \to L^q(\R^3)$ independently of $\delta$ with $p$, $q$ as in Proposition \ref{prop:FourierRestrictionExtension}. For this it is enough to show restricted weak type bounds
\begin{equation*}
\| A_\delta \|_{L^{q_0,\infty}} \lesssim \| f \|_{L^{p_0,1}}
\end{equation*}
for $(1/p_0,1/q_0) = (61/70,3/10)$ and the bounds
\begin{equation*}
\| A_\delta f \|_{L^q} \lesssim \| f \|_{L^{p,1}}
\end{equation*}
for $(1/p,1/q) \in ((61/70,3/10),(1,3/10)]$, since strong bounds for $A_\delta$ with $p,q$ as in Proposition \ref{prop:FourierRestrictionExtension} are recovered by interpolation and duality. As $\nabla p(\xi) \neq 0$ for $\xi \in \text{supp} (\beta_1)$ by construction, we can change to generalized polar coordinates. Let $\xi = \xi(p,q)$, where $p$ and $q$ are complementary coordinates.\\
Write
\begin{equation*}
A_\delta f(x) = \int \frac{e^{ix\cdot\xi} \beta_1(\xi)}{p(\xi) + i\delta} \hat{f}(\xi) d\xi = \int dp \int dq \frac{e^{ix\cdot\xi(p,q)} \beta(\xi(p,q)) h(p,q) \hat{f}(\xi(p,q))}{p + i \delta},
\end{equation*}
where $h$ denotes the Jacobian. We can suppose that $|\partial^\alpha h | \lesssim_\alpha 1$ choosing $\text{supp}( \beta)$ small enough. The expression is estimated as in \cite[Subsection~5.2]{MandelSchippa2021} by suitable decompositions in Fourier space and crucially depending on the Fourier restriction estimates for Proposition \ref{prop:FourierRestrictionExtension}; see \cite{MR4076079} for $p(\xi) = |\xi|^\alpha$. We write
\begin{equation*}
\frac{1}{p(\xi) + i \delta} = \frac{p(\xi)}{p^2(\xi) + \delta^2} - i \frac{\delta}{p^2(\xi) + \delta^2} = \mathfrak{R}(\xi) - i \mathfrak{I}(\xi).
\end{equation*}
As in \cite{MandelSchippa2021}, $\mathfrak{I}(D)$ is estimated by Minkowski's inequality and Fourier restriction--extension estimates, in the present context from Proposition \ref{prop:FourierRestrictionExtension}. The only difference in the estimate of $\mathfrak{R}(D)$ is that \cite[Lemma~5.1]{MandelSchippa2021} is applied for $k=\frac{3}{2}$. For details we refer to \cite[Section~4]{MandelSchippa2021}. This finishes the proof of the estimate for $A_\delta$.

For the estimate of $B_\delta$, we carry out a further decomposition in Fourier space: By ellipticity, there is $\alpha > 0$ and $R\geq 1$ such that
\begin{equation*}
|p(\xi)| \gtrsim |\xi|^\alpha
\end{equation*}
provided that $|\xi| \geq R$. Let $\beta_2(\xi) = \beta_{21}(\xi) + \beta_{22}(\xi)$ with $\beta_{21}, \beta_{22} \in C^\infty$ and $\beta_{22}(\xi) = 0 $ for $|\xi| \leq R$, $\beta_{22}(\xi) = 1$ for $|\xi| \geq 2R$.\\
We can estimate
\begin{equation*}
\| B_\delta (\beta_{21}(D) f) \|_{L^q} \lesssim \| f \|_{L^p}
\end{equation*}
for any $1 \leq p \leq q \leq \infty$ by Young's inequality uniform in $\delta$. This gives no additional assumptions on $p$ and $q$. We estimate the contribution of $\beta_{22}$ by properties of the Bessel kernel (cf. \cite[Theorem~30]{CossettiMandel2020}) 
\begin{equation*}
\| B_\delta (\beta_{22}(D)  f) \|_{L^q(\R^3)} \lesssim \| \beta_{22}(D) f \|_{L^p(\R^3)}
\end{equation*}
for $1 \leq p,q \leq \infty$ and $0 \leq \frac{1}{p} - \frac{1}{q} \leq \frac{\alpha}{3}$ with the endpoints excluded for $\alpha \leq 3$. For $\alpha > 3$ this estimate holds true for $1 \leq p \leq q \leq \infty$. This corresponds to the second assumption on $p$ and $q$ in Theorem \ref{thm:LimitingAbsorptionPrinciple}. Lastly, we give the standard argument for constructing solutions: For $\delta > 0$, consider the approximate solutions $u_\delta \in L^q(\R^3)$
\begin{equation*}
\hat{u}_\delta(\xi) = \frac{\hat{f}(\xi)}{p(\xi) + i \delta}.
\end{equation*}
By the above, we have uniform bounds
\begin{equation*}
\| u_\delta \|_{L^q(\R^3)} \lesssim \| f \|_{L^{p_1}(\R^3) \cap L^{p_2}(\R^3)}.
\end{equation*}
By the Banach--Alaoglu--Bourbaki theorem, we find a weak limit $u_\delta \to u$, which satisfies the same bound. We observe that
\begin{equation*}
P(D) u_\delta = f - i \frac{\delta}{P(D)+i \delta} f.
\end{equation*}
Since
\begin{equation*}
\| \frac{\delta}{P(D)+i \delta} f \|_{L^q} \lesssim \delta \| f \|_{L^{p_1} \cap L^{p_2}},
\end{equation*}
we find that $P(D) u_\delta \to f$ in $L^q(\R^3)$. Since $P(D) u_\delta \to P(D) u$ in $\mathcal{S}'(\R^3)$, this shows that
\begin{equation*}
P(D) u = f
\end{equation*}
in $\mathcal{S}'(\R^3)$. The proof is complete. \hfill $\Box$

\subsection{Spectral theory of Schr\"odinger operators on $\Z^3$}
\label{Subsection:discrete}

As a Fourier multiplier, for $\xi \in \T^3=(\R / (2\pi \Z))^3$, the discrete Laplacian acts as
\begin{align}\label{dispersion relation lattice}
-\frac{1}{2}\widehat{(\Delta u)}(\xi)=p(\xi)\widehat{u}(\xi),\quad p(\xi)=\sum_{j=1}^3[1-\cos(\xi_j)],
\end{align}
where $\widehat{u}$ is the discrete Fourier transform of $u$, 
\begin{align}\label{discrete FT}
\widehat{u}(\xi)=\sum_{x\in \Z^3}
e^{-i x\cdot \xi}u(x),\quad \xi\in \mathbb{T}^3.
\end{align}
Hence $\sigma(H)=\set{p(\xi)}{\xi\in \mathbb{T}^3}=[0,6]$.
The set of critical values of $p$ is 
\begin{align*}
\set{p(\xi)}{\nabla p(\xi)=0}=\{0,2,4,6\}.
\end{align*}
In addition to the critical values, the value $a=3$ plays a special role. This is because the level set $p(\xi)=3$ has a different type of degeneracy (flat umbilic, see \cite{MR2324803}). 

Erd\H{o}s--Salmhofer \cite[Section 4]{MR2324803} proved that $p(\xi)$ satisfies their assumptions 1-4. These assumptions imply that the level sets $\Sigma_a:=\{p(\xi)=a\}$ are in the class $\mathcal{C}$ for $a\notin \{0,2,3,4,6\}$; in fact, their assumption $3$ is not needed for this conclusion. Therefore, our decay bound \eqref{eq: 3/4 decay} strengthens their result in two ways: First, the decay is uniform in all directions and second, it is valid under the relaxed assumptions stated above.


\begin{proof}[Proof of Theorem \ref{theorem resolvent discrete}]
The estimates (2) and (3) are standard consequences of the decay bound \eqref{eq: 3/4 decay}, see e.g.\ \cite[Theorem 1.2 (i), (ii)]{MR4153099}. The spectral measure estimate (1) is an immediate consequence of (2).  
\end{proof}

\begin{proof}[Sketch of the proof of Theorem \ref{wave operators discrete}]
The deterministic result is a standard consequence of the uniform resolvent estimate (using the method of smooth perturbations, see e.g.\ \cite[Corollary to Theorem XII.31]{MR0493421} or \cite[Chapter 4, \textsection
5, Corollary 7]{MR1180965}). For random potentials the result follows from Bourgain's proof in \cite{MR2083389}. One just replaces the Stein-Tomas bound (2.4) there with the extension estimate
\begin{align}\label{extension operator discrete Laplacian}
\|\mathcal{E}g\|_{\ell^{14/3}(\Z^3)}\lesssim \|g\|_{L^2(\Sigma)},
\end{align}
where $\mathcal{E}$ is the Fourier extension operator corresponding to the surface $\Sigma_a$, $a\in J$. This bound is equivalent to the spectral measure estimate in Theorem \ref{theorem resolvent discrete}, by a $TT^*$ argument. We briefly sketch Bourgain's  main arguments. The first observation is that it suffices to prove the claim for small perturbations, i.e.\ for $V$ replaced by $\kappa V$ for some small coupling $\kappa$. The reason is that $V$ can always be split into a short range (compactly supported) part $V_S$ and a long range part $V_L$. The short range part is of finite rank (since we are on a lattice), and one may use the abstract trace class method (i.e.,\ the Kato--Rosenblum theorem, see e.g.\ \cite[Chapter 4, \textsection
2, Theorem 1 and Corollary 2]{MR1180965}) to establish the existence and completeness of the wave operators $W_{\pm}(H_0+V_S,H_0+V)$. Hence, by the chain rule for wave operators (see e.g.\ \cite[Section 2.1]{MR1180965}) it remains to prove that $W_{\pm}(H_0,H_0+V_L)$ exist and are complete. This achieved by the smooth method. Due to the decay assumption, the long range part can be made arbitrarily small, and one could appeal to \cite[Chapter 4, \textsection 6, Theorem 1]{MR1180965}) in the deterministic case. In the random case, Bourgain shows that the Born series for the perturbed resolvent converges with high probability and establishes a limiting absorption principle \cite[Formula (2.5)]{MR1877824}. This already implies that $H$ has only a.c.\ spectrum in $J$. The key tools in the proof of the limiting absorption principle are entropy bounds, which imply estimates on the expectation values of norms of certain 'elementary operators'. Representing the unitary groups $e^{itH}$ and $e^{itH_0}$ in terms of the resolvent, the bounds for the 'elementary operators' also yield the statement about the wave operators. The 'elementary operators' feature in a multilinear expansion of the resolvent and are related to the Fourier restriction operator. In \cite{MR2083389} Bourgain uses the fact that the surfaces $\{p(\xi)=\lambda\}$ (where $p(\xi)=2-\cos(\xi_1)-\cos(\xi_2)$ is the symbol of the discrete Laplacian) are curved away from the edges of the spectrum (see also \cite[Lemma 3.3]{MR1979773}). By the Stein--Tomas theorem, this implies the $L^2(\Sigma)\to \ell^{6}(\Z^2)$ analog of \eqref{extension operator discrete Laplacian}. This bound accounts for the $\ell^3$ norm of the potential in \cite{MR2083389} (see inequality (3.14) there). Note that $1/p-1/p'=2/3$ for $p'=6$, i.e.,\ $3$ is twice the exponent that one would get in the deterministic case in two dimensions. In three dimensions, \eqref{extension operator discrete Laplacian} gives $1/p-1/p'=4/7$ for $p'=14/3$, which explains why the deterministic result holds for $V\in \ell^{7/4}$. The randomization arguments of Bourgain allow us to double the exponent (up to some technical $\epsilon$ loss in the decay, as stated in the assumptions of the theorem). 
\end{proof}

\bibliographystyle{plain}

\end{document}